\documentclass[11pt]{amsart}

\usepackage[T1]{fontenc}
\usepackage[utf8]{inputenc}
\usepackage{lmodern}
\usepackage{amsmath,amssymb,mathtools,mathrsfs}
\usepackage{enumitem}
\usepackage{xcolor}
\usepackage{graphicx}
\usepackage[colorlinks=true,linkcolor=blue!55!black,citecolor=blue!55!black,urlcolor=blue!55!black]{hyperref}
\usepackage[margin=2.5cm]{geometry}

\numberwithin{equation}{section}

\newtheorem{theorem}{Theorem}

\newtheorem{lemma}[theorem]{Lemma}
\newtheorem{corollary}[theorem]{Corollary}
\theoremstyle{definition}

\newtheorem*{definition*}{Definition}
\theoremstyle{remark}
\newtheorem{remark}[theorem]{Remark}

\usepackage[nameinlink,noabbrev]{cleveref}
\crefname{proposition}{proposition}{propositions}
\Crefname{proposition}{Proposition}{Propositions}
\crefname{lemma}{lemma}{lemmas}
\Crefname{lemma}{Lemma}{Lemmas}
\crefname{corollary}{corollary}{corollaries}
\Crefname{corollary}{Corollary}{Corollaries}
\crefname{definition}{definition}{definitions}
\Crefname{definition}{Definition}{Definitions}
\crefname{remark}{remark}{remarks}
\Crefname{remark}{Remark}{Remarks}

\newcommand{\R}{\mathbb R}
\newcommand{\T}{\mathbb T}
\newcommand{\eps}{\varepsilon}

\newcommand{\Id}{\mathrm{Id}}
\newcommand{\Int}{\operatorname{int}}
\newcommand{\dist}{\operatorname{dist}}

\newcommand{\Div}{\operatorname{div}}

\newcommand{\vol}{\operatorname{vol}}
\newcommand{\norm}[1]{\left\lVert #1\right\rVert}
\newcommand{\abs}[1]{\left\lvert #1\right\rvert}

\title[Stability and instability in magnetic fields]{Stability and instability
for particle dynamics\\ in magnetic fields}
\author{Alberto Enciso}
\author{Daniel Peralta-Salas}
\date{}

\begin{document}

\begin{abstract}
We study zero-speed equilibria of the Newton--Lorentz equation $\ddot x=\dot x\times B(x)$ in $\R^3$, which describes the motion of a charged particle in a magnetic field $B$. We prove Lyapunov instability wherever the field $B$ does not vanish (under $C^1$ regularity), and at every isolated zero of $B$ provided it is real-analytic. These results do not assume that the field is divergence-free. In contrast, we construct a compactly supported divergence-free $C^\infty$ field that is flat at an isolated zero and makes the corresponding equilibrium Lyapunov stable.
\end{abstract}

\maketitle

\section{Introduction}
\label{sec:introduction}

Our objective in this paper is to study the Lyapunov stability of a charged
particle at rest in a three-dimensional magnetic field.  In order to state the mathematical formulation of this problem, let us introduce some notation. Let $U\subset\R^3$ be a
neighborhood of the origin and let $B:U\to\R^3$ be a $C^1$ vector field (we do not need to assume that it is divergence free). A unit-charge, unit-mass particle under the influence of the magnetic field $B$ obeys the Newton--Lorentz equation of motion
\begin{equation}
    \ddot x=\dot x\times B(x)\,,
    \label{eq:magnetic-equation}
\end{equation}
where, as usual, the symbol $\cdot$ denotes time derivative. This equation is solved for each initial condition $x_0\in U$ and $\dot x_0\in \mathbb R^3$, which yields a solution $x(t;x_0,\dot x_0)$. All along this article we will usually omit the dependence of the solution with the initial condition, and will simply write $x(t)$. We recall that the phase space is parametrized by coordinates $(x,\dot x)\in \mathbb R^6$.

Concerning the equilibria of this dynamical system, we observe that for every point $p\in U$, the phase-space point $(p,0)$ is an equilibrium of the Newton--Lorentz equation~\eqref{eq:magnetic-equation}. Accordingly, in phase space the set of equilibria consists of the three-dimensional set $U \times \{0\}$. By translation, we may take $p=0$ without any loss of generality, so in what follows we will focus on the (non-isolated) equilibrium point~$(0,0)$. We will see that it is convenient to distinguish between the case that $0$ is a \emph{magnetic zero} (i.e., $B(0)=0$) and the case that it is a regular point of $B$ (i.e., $B(0)\neq 0$).

We use the following formulation of Lyapunov (in)stability:

\begin{definition*}
The equilibrium $(0,0)$ is \emph{Lyapunov unstable} if there is $r>0$
such that, for every $\delta>0$, one can find initial conditions $(x_0,\dot x_0)$ satisfying
\[
    \abs{x_0}+\abs{\dot x_0}<\delta\,,
\]
and a time $T>0$ for which $\abs{x(T)}\ge r$.  Otherwise it is
\emph{Lyapunov stable}.
\end{definition*}

This is equivalent to the usual definition of Lyapunov (in)stability because the Lorentz force
does no work. Therefore, the energy of a solution $x(t)$ consists only of the kinetic part $\frac12|\dot x(t)|^2$, and the speed $|\dot x(t)|$ is a conserved quantity of the motion. Every solution then satisfies
\begin{equation}
    \abs{\dot x(t)}=\abs{\dot x(0)}.
    \label{eq:speed-conservation}
\end{equation}
Thus a phase-space escape solution $(x(t),\dot x(t))$ with initial speed $|\dot x_0|$ tending to zero, necessarily contains a spatial escape $|x(t)|$.

Our first result shows that an equilibrium is unstable whenever the magnetic
field does not vanish.  This holds under only $C^1$ regularity. The proof is elementary and is given in
Appendix~\ref{sec:nonzero-field}.

\begin{theorem}
\label{thm:away-from-magnetic-zeros}
Let $B\in C^1(U;\R^3)$.  If $B(0)\ne0$, then the zero-speed
equilibrium $(0,0)$ of \eqref{eq:magnetic-equation} is Lyapunov
unstable.
\end{theorem}

We next focus on magnetic zeros, the only points where stability
can occur.  The problem is
more delicate there, and the regularity of $B$ becomes essential.
For real-analytic fields, we prove that an isolated magnetic zero
still gives an unstable zero-speed equilibrium.

\begin{theorem}
\label{thm:analytic-isolated}
Let $B$ be real-analytic near $0\in\R^3$, with $B(0)=0$, and suppose
that $0$ is an isolated magnetic zero.  Then $(0,0)$ is Lyapunov unstable.
No assumption on $\Div B$ (the divergence of the vector field $B$) is required.
\end{theorem}

The key idea of the proof, given in \Cref{sec:tracking}, is to compare
low-energy trajectories (equivalently, low-speed trajectories) of the second-order Newton--Lorentz equation
\eqref{eq:magnetic-equation} with magnetic field lines, which solve
the first-order ODE 
$$\dot q=B(q)\,.$$  
We prove that particles launched tangent to a
nonvanishing field line, in either direction, track any fixed finite
segment (after a suitable rescaling of time) as their speed tends to zero.
The second ingredient is the remarkable Brunella's theorem~\cite{Brunella1998},
which shows the existence of a field line converging to an isolated zero for any real-analytic vector field in $\mathbb R^3$.
Applying the comparison to finite arcs directed away from the zero
produces particle trajectories that start arbitrarily close to
$(0,0)$ and escape from a fixed neighborhood.

The $C^\infty$ category is radically different: even an isolated magnetic
zero can define a Lyapunov-stable equilibrium.

\begin{theorem}
\label{thm:smooth-counterexample}
There exists a divergence-free field $B\in C^\infty_c(\R^3;\R^3)$,
satisfying the flatness condition $D^\alpha B(0)=0$ for all multiindices
$\alpha$ but with $B(x)\neq0$ for all $0<\abs x<1$, such that
$(0,0)$ is a Lyapunov-stable equilibrium of \eqref{eq:magnetic-equation}.
\end{theorem}

The key idea of the proof, given in \Cref{sec:counterexample}, is to
place magnetic barriers around the origin at successively smaller
scales.  Each barrier is a toroidal
collar where a conserved angular momentum prevents outward
crossing below a positive speed threshold.  Rapidly decreasing
amplitudes make the field smooth and flat at the origin, while a
barrier within each neighborhood gives Lyapunov stability.  To keep
the zero isolated, we adapt the nested-torus geometry of Kozlov and
Treshchev~\cite{KozlovTreshchev1999}, filling the regions between barriers with nonvanishing
divergence-free fields.
A consequence of \Cref{thm:smooth-counterexample} is that even the
infinite jet of the magnetic field at a point does not determine
the Lyapunov stability of the corresponding equilibrium (obviously, if $B\equiv 0$ in a neighborhood of the origin, the equilibrium $(0,0)$ is Lyapunov unstable).

These results can be compared with the classical stability problem
for potential systems $$\ddot x=-\nabla V(x)\,.$$  The
Lagrange--Dirichlet theorem gives stability at strict local minima,
but its converse is false for smooth potentials~\cite{Laloy1976}.  For real-analytic
potentials on $\R^2$, Palamodov proved that a critical point that is
not a local minimum defines an unstable equilibrium~\cite{Palamodov1977}.  For real-analytic potentials on $\R^3$, the
corresponding assertion remains open; see~\cite{BurgosPaternain2022Converse} for partial results.

Relatively little was previously known in general about the stability
of dynamics in magnetic fields.  Particle motion near specific magnetic
zeros, including cusp fields,
was studied in~\cite{JaroensutasineeRowlands2000}.  Bolotin and
Negrini~\cite{BolotinNegrini1995} studied asymptotic motions in
Lagrangian systems with gyroscopic forces.  Burgos and Paternain
\cite{BurgosPaternain2022} gave an instability criterion that applies,
in particular, to the motion of charged particles in magnetic fields (the assumption is that the magnetic field must satisfy the integrability condition $B\cdot \text{curl } B=0$ in a neighborhood of the equilibrium; this is the same as $B$ being gradient-like).  For systems with a nonconstant
analytic potential and a magnetic field, Burgos~\cite{Burgos2026} obtained further instability criteria under a weak assumption.

\section{Lyapunov instability of analytic magnetic fields with isolated zeros}
\label{sec:tracking}

To prove \Cref{thm:analytic-isolated}, we shall use a magnetic field
line converging to the origin to construct particle trajectories that
escape from a fixed neighborhood.  The key ingredient of the proof is a
comparison between slow particle trajectories and fixed nonvanishing
field-line segments. The conclusion of this comparison is that the smaller the speed of the particle the closer the trajectory of the particle to the fixed piece of the magnetic line (if the initial condition is a point of the magnetic line and the velocity is tangent to this line).  We then apply it to finite field-line arcs followed away
from the zero, choosing the particle speed separately for each arc.

Let $x(t)$ be a solution of the Newton--Lorentz equation~\eqref{eq:magnetic-equation} with constant speed
$|x(t)|=\eps>0$, and introduce a slow time variable $s=\eps t$:
\[
    X(s):=x(s/\eps),
    \qquad
    V(s):=X'(s)=\frac{1}{\eps} \dot x(s/\eps)\,,
\]
where $'$ denotes derivative with respect to~$s$. Then the pair $(X(s),V(s))$ satisfies the ODEs
\begin{equation}
    X'=V\,,
    \qquad
    \eps V'=V\times B(X)\,,
    \qquad
    \abs V=1\,.
    \label{eq:slow-system}
\end{equation}
Conversely, if $X(s)$ is a unit-speed curve (i.e., $|X'|=1$) with curvature vector
$K:=X''$, then the equation
\begin{equation}
    \eps K=X'\times B(X)
    \label{eq:magnetic-path}
\end{equation}
implies that $x(t):=X(\eps t)$ is a solution of the Newton--Lorentz equation~\eqref{eq:magnetic-equation} with speed $|\dot x(t)|=\eps$. Accordingly, up to the scaling $s=\eps t$, both the Newton--Lorentz equation~\eqref{eq:magnetic-equation} and Equation~\eqref{eq:magnetic-path} (or Equation~\eqref{eq:slow-system}) are equivalent.

The following lemma gives the aforementioned comparison estimate for particles launched
tangent to a magnetic field-line. This is related to the guiding-center
approximation; see \cite{FilbetRodrigues2020} for a rigorous treatment
in three dimensions.

\begin{lemma}
\label[lemma]{prop:field-tracking}
Let $B\in C^2(U;\R^3)$ be a vector field, and let
$\Gamma:[0,S]\to U$ be a unit-speed curve tangent to a field-line, i.e., 
\[
    \Gamma'(s)=\sigma\frac{B(\Gamma(s))}{\abs{B(\Gamma(s))}},
    \qquad
    \sigma\in\{+1,-1\},
\]
with $B(\Gamma(s))\ne0$ on $[0,S]$.  Then there are constants
$C_\Gamma,\eps_\Gamma>0$, depending on $B$ and the fixed segment
$\Gamma$ but not on $\eps$, such that for every
$0<\eps<\eps_\Gamma$ the solution of the slow system~\eqref{eq:slow-system} with
\[
    X(0)=\Gamma(0),
    \qquad
    V(0)=\Gamma'(0)
\]
exists on $[0,S]$ and satisfies
\begin{equation}
 \sup_{0\le s\le S}
 \bigl(\abs{X(s)-\Gamma(s)}+
       \abs{V(s)-\Gamma'(s)}\bigr)
 \le C_\Gamma\eps.
 \label{eq:tracking-estimate}
\end{equation}
\end{lemma}

\begin{proof}
Choose an open set $U_0\Subset U$ such that
$\Gamma([0,S])\subset U_0$, $B$ does not vanish on $\overline{U_0}$, and
\[
 d_\Gamma
 :=\dist\bigl(\Gamma([0,S]),\partial U_0\bigr)>0.
\]
For fixed $\eps>0$, let $\tau_\eps\le S$ be the supremum of the times on
which the solution $X(s)$ of~\eqref{eq:slow-system} with the aforementioned initial conditions exists and remains in $U_0$.  On $U_0$ set
\[
    b:=\abs B\,,
    \qquad
    n:=\frac B b\,.
\]
All along this proof, fields and functions without an explicit argument are
evaluated at $X(s)$, and all the equations and inequalities hold on the whole interval $s\in[0,\tau_\eps)$.  On $[0,\tau_\eps)$ we define
\[
 a:=V\cdot n(X)\,,
 \qquad
 w:=V-a n(X)\,,
 \qquad
 P:=\Id-n\otimes n\,,
\]
which will allow us to compare the tangential and perpendicular components of the velocity $V(s)$ with respect to the field-line $\Gamma(s)$. In particular, $w\perp n$ and 
$$a^2+\abs w^2= a^2+1-2a^2 + a^2=1\,.$$  
On the normal plane to the curve $\Gamma(s)$ (for each $s$) define
\[
    Rq:=q\times n\,,
    \qquad
    K:=(n\cdot\nabla)n\,,
    \qquad
    D_s^\perp q:=Pq'\,,
\]
for $q(s)\in\mathbb R^3$. We use the same notation $K$ as for the curvature vector because $(n\cdot \nabla)n$ evaluated at $\Gamma(s)$ is precisely $\Gamma''(s)$.
Direct differentiation of Equation~\eqref{eq:slow-system} gives
\begin{equation}
 D_s^\perp w
 =\frac b{\eps}Rw-a^2K-aP(Dn)w\,,
 \label{eq:normal-w}
\end{equation}
where we have used that $P(Dn)V=aK+P(Dn)w$.

Next, set
\[
    w_0:=-\frac{\eps a^2}b RK.
\]
Since $R^2=-\Id$ on $n^\perp$ and $K\cdot n=0$, we obtain
\[
    \frac b{\eps}Rw_0=a^2K\,.
\]
Moreover,
\[
    a'=V\cdot(Dn)V,
\]
because $\eps V'\cdot n=(V\times B)\cdot n=0$.  Hence $a'$ is uniformly bounded, and
the $C^2$ regularity of $B$ together with the fact that $|a|\leq 1$, give the bound
\[
    \abs{w_0}+\abs{D_s^\perp w_0}\le C\eps\,.
\]
Now, for $z:=w-w_0$, Equation~\eqref{eq:normal-w} becomes
\[
 D_s^\perp z
 =\frac b{\eps}Rz-aP(Dn)z+F_\eps\,,
\]
where the term $F_\eps$ is bounded as $\abs{F_\eps}\le C\eps$. Noticing that the initial conditions of the trajectory $X(s)$ lead to $w(0)=0$, we obtain that
$z(0)=-w_0(0)$ and hence $\abs{z(0)}\le C\eps$.  Since $z\perp n$, we infer
\begin{equation}
 \frac12\frac{d}{ds}\abs z^2
 =z\cdot D_s^\perp z
 =-a\,z\cdot(Dn)z+z\cdot F_\eps
 \le C\abs z^2+C\eps \abs z
 \le C\abs z^2+C\eps^2\,,
 \label{eq:tracking-energy}
\end{equation}
with an $\eps$-independent constant that may differ from line to line.
Here the harmful $\eps^{-1}$ term has disappeared because of the identity
$z\cdot Rz=0$. Then, Gronwall's inequality applied to Equation~\eqref{eq:tracking-energy} gives, uniformly on $[0,\tau_\eps)$,
\[
    \abs z+\abs w\le C\eps.
\]
For small $\eps$, the identity $a^2+\abs w^2=1$ prevents $a$ from
vanishing.  Since $a(0)=\sigma\in\{+1,-1\}$, the previous bound of $|w|$ and this identity give
$$a=\sigma+O(\eps^2)\,.$$  Consequently, using the previous estimates, 
\[
    X'=V=w+an=\sigma n(X)+O(\eps)\,.
\]
Since $n$ is Lipschitz on $\overline{U_0}$ and $X(0)-\Gamma(0)=0$, comparison with
$\Gamma'=\sigma n(\Gamma)$ and another application of Gronwall's
inequality, readily yield
\[
 \abs{X(s)-\Gamma(s)}
 +\abs{V(s)-\Gamma'(s)}
 \le C_\Gamma\eps\,,
\]
for all $0\le s<\tau_\eps$.
Choosing a small $\eps_\Gamma$ so that
$C_\Gamma\eps_\Gamma<d_\Gamma/2$, we conclude that $X(s)$ cannot exit through $\partial U_0$ (otherwise, it would contradict the definition of $d_\Gamma$).
All the preceding estimates hold for any $0<\eps < \eps_\Gamma$, and if $\tau_\eps<S$, the solution $X(s)$ and the estimates extend by continuity to
$s=\tau_\eps$. Thus $\tau_\eps=S$, proving the bound~\eqref{eq:tracking-estimate}.  The constants depend only on $S$, the
chosen set $U_0$, a positive lower bound for $\abs B$ there, and the
corresponding $C^2$ bounds for $B$.
\end{proof}

The first consequence of the comparison in Lemma~\ref{prop:field-tracking} is that escape along finite
field-line segments implies instability of the particle equilibrium.

\begin{corollary}
\label[corollary]{cor:field-segment}
Suppose that there are $r>0$, points $p_j\to0$, finite numbers
$S_j>0$, and unit-speed nonvanishing field-line segments
$\Gamma_j:[0,S_j]\to U$ such that
\[
 \Gamma_j(0)=p_j,
 \qquad
 \Gamma_j(S_j)\in\partial\mathbb B_r(0),
 \qquad
 \Gamma_j'
 =\sigma_j\frac{B(\Gamma_j)}{\abs{B(\Gamma_j)}},
 \quad
 \sigma_j\in\{+1,-1\}.
\]
Then $(0,0)$ is Lyapunov unstable.
\end{corollary}

\begin{proof}
Given $\delta>0$, choose $j$ such that $\abs{p_j}<\delta/2$.  Apply
\Cref{prop:field-tracking} to this fixed $j$-th segment, and denote the  constants that appear in the estimate of the lemma
by $C_j$ and $\eps_j^*$.  Choose
\begin{equation}
 0<\eps_j<
 \min\left\{\frac\delta2,\eps_j^*,\frac r{2C_j}\right\}.
 \label{eq:segment-dependent-speed}
\end{equation}
The corresponding slow trajectory $X_j(s)$ provided by Lemma~\ref{prop:field-tracking} satisfies
\[
 \abs{X_j(S_j)}
 \ge \abs{\Gamma_j(S_j)}-C_j\eps_j
 \ge \frac r2.
\]
The slow solution $X_j(s)$ yields a solution $x_j(t)=X_j(\eps_j t)$ to the Newton--Lorentz equation~\eqref{eq:magnetic-equation} with initial condition $x_j(0)=p_j$ and $\abs{\dot x_j(0)}=\eps_j$. Therefore, the initial condition in phase space has size
$\abs{p_j}+\eps_j<\delta$, and the escape occurs at the finite physical
time
\[
    T_j=\frac{S_j}{\eps_j}.
\]
All these quantities may depend on the chosen segment. Since $\delta$ can be any real number, and $r$ is fixed, the instability follows.
\end{proof}

The final ingredient in the proof of \Cref{thm:analytic-isolated} is Brunella's theorem on the existence of integral curves asymptotic to isolated zeros of analytic vector fields in $\mathbb R^3$. We state it for the sake of completeness:

\begin{theorem}[Brunella~\cite{Brunella1998}]
\label{thm:Brunella}
Let $Y$ be a real-analytic vector field near $0\in\R^3$.  If $0$ is an
isolated zero of $Y$, then $Y$ has a nonconstant integral curve $q(t)$
such that $q(t)\to0$ as $t\to+\infty$ or as $t\to-\infty$.
\end{theorem}

Now we can apply Corollary~\ref{cor:field-segment} by extracting finite field-line segments from an integral curve of $B$ converging to the origin provided by Brunella's theorem.

Indeed, let $q(t)$ be a (nonconstant) integral curve of $B$
converging to the origin as $t\to+\infty$ or as $t\to-\infty$. Since $B$ is real-analytic and the origin is an isolated zero of $B$, such a curve exists by Theorem~\ref{thm:Brunella}. Fix a regular point
$q_*=q(t_*)$.  Choose $0<r<\abs{q_*}$ and an increasing (in absolute value) sequence of times $t_j$ (with $\abs{t_j}>\abs{t_*}$) such that $q_j:=q(t_j)\to0$ and $\abs{q_j}<r$.
For each $j$, follow the finite orbit arc $q(t)$ from $q_j$ toward $q_*$ and
truncate it at the first point encountered on
$\partial\mathbb B_r(0)$.  By construction, the constructed arc is compact and $\abs B$ is nonvanishing at any of its points.  Its arclength, call it $S_j$, is finite because it is the image of a compact time interval by a smooth map.  Reparametrize this truncated arc by arclength.
If $q(t)\to0$ as $t\to+\infty$, the resulting segment is followed against
$B$; if $q(t)\to0$ as $t\to-\infty$, it is followed with $B$. This gives a countable family of arcs $\Gamma_j$ as in
\Cref{cor:field-segment}. The application of this corollary then proves Lyapunov instability. This completes the proof of Theorem~\ref{thm:analytic-isolated}.

\section{A smooth divergence-free magnetic field giving rise to a Lyapunov stable equilibrium}
\label{sec:counterexample}

To prove \Cref{thm:smooth-counterexample}, we need toroidal magnetic barriers
enclosing successively smaller neighborhoods of a point.  Each barrier
needs only to stop particles below its own positive speed threshold.
The task is to fit these barriers into a smooth divergence-free vector field
without introducing any further zeros.  The method of proof consists in first building one toroidal
barrier, and then use the nested-torus geometry of Kozlov and Treshchev
\cite[Section~3]{KozlovTreshchev1999} to join it to a smaller copy.
Iteration with rapidly decreasing amplitudes gives smoothness, while
the retained nested barriers give stability.  For clarity, we divide the
argument into four steps.

\subsubsection*{Step 1: Toroidal magnetic barriers}

The barrier comes from a conserved angular momentum with two terms: one
linear in the velocity, the other a function of position.  If the
latter takes different values on the two boundaries of the collar (which is an annular toroidal shell), sufficiently
slow particles cannot cross the collar from the inner boundary to the outer one: the velocity term is too
small to compensate for that difference.  To state the criterion, write
\[
    \beta_B:=\iota_B\vol,
    \qquad
    \vol:=dx_1\wedge dx_2\wedge dx_3.
\]
Thus $\Div B=0$ if and only if $d\beta_B=0$. The following lemma proves the key property in order to construct dynamical barriers when the particle's speed is small.

\begin{lemma}
\label[lemma]{lem:noether-barrier}
Let $W^-\Subset W^+$ be bounded open regions of $\mathbb R^3$ such that
$\mathcal C:=\overline{W^+}\setminus W^-$ is a compact toroidal collar bounded
by $\Sigma_\pm:=\partial W^\pm$.  Suppose that there are a Euclidean Killing
field $K$ and a smooth function $\psi$ on a neighborhood of
$\mathcal C$, satisfying
\[
    d\psi=-\iota_K\beta_B,
    \qquad
    \psi|_{\Sigma_\pm}=\psi_\pm,
    \qquad
    \Delta:=\abs{\psi_+-\psi_-}>0\,,
\]
for some constants $\psi_\pm$. Set $M:=\max_{\mathcal C}\abs K$.  A solution $x(t)$ of the Newton--Lorentz equation~\eqref{eq:magnetic-equation} starting in $W^-$ cannot leave $W^+$ if
its speed is bounded as
$$|\dot x(t)|<\frac{\Delta}{2M}\,.$$
\end{lemma}
\begin{remark}
 In the vectorial language the condition $d\psi=-\iota_K\beta_B$ is equivalent to $K\times B=\nabla \psi$. This implies that the function $\psi$ is a first integral of both $B$ and $K$, and since $\psi$ is constant on each component $\Sigma_\pm$, the fields $B$ and $K$ are tangent to these surfaces. Moreover, in the case that $B$ is assumed to be divergence-free, $K$ is a Killing symmetry of $B$, i.e., $[B,K]=0$. 
\end{remark}
\begin{proof}
Let $x(t)$ be a solution to the Newton--Lorentz equation~\eqref{eq:magnetic-equation} with magnetic field $B$ as in the statement, and satisfying the stated speed bound. We use the notation $v=\dot x$ for the velocity field.  Since $K$ is Killing, its Jacobian matrix $DK$ is antisymmetric, and hence $v\cdot DK(x)v=0$.  Therefore, if the
trajectory $x(t)$ is assumed to be contained in $\mathcal C$, we obtain
\[
    \frac{d}{dt}\bigl(v(t)\cdot K(x(t))\bigr)
    =(v(t)\times B(x(t)))\cdot K(x(t))
    =\vol(B,K,v)|_{x(t)}
    =\beta_B(K,v)|_{x(t)}\,.
\]
It follows that
\[
\frac{d}{dt}\Big(v(t)\cdot K(x(t))+\psi(x(t))\Big)=0\,,
\]
and therefore
\begin{equation}
    v(t)\cdot K(x(t))+\psi(x(t))
    \quad\text{is constant while the trajectory stays in }\mathcal C\,.
    \label{eq:noether-momentum}
\end{equation}
If the trajectory leaves $W^+$, take its first contact $p_1$ with
$\Sigma_+$ and its last preceding contact $p_2$ with $\Sigma_-$. The
segment between these contacts lies in $\mathcal C$.  Comparing~\eqref{eq:noether-momentum} at the points $p_1$ and $p_2$, we infer that
\[
v(p_1)\cdot K(p_1)+\psi_-=v(p_2)\cdot K(p_2)+\psi_+\,,
\]
and hence
\[
\Delta=|v(p_1)\cdot K(p_1)-v(p_2)\cdot K(p_2)|\leq 2|v|M\,,
\]
where we have used the speed conservation, i.e., $|v(p_1)|=|v(p_2)|=|v|$.
We conclude from this inequality that $\abs v\geq \frac{\Delta}{2M}$, which contradicts the assumed speed bound on the solution $x(t)$. This shows that $\Sigma_+$ is a magnetic barrier for the trajectory $x(t)$, thus completing the proof of the lemma.
\end{proof}

To iterate this criterion, we need a spatial region with two boundary
properties: an outer toroidal collar with rotational symmetry, which provides
the barrier, and an inner torus to which the field is tangent, where
we can attach the next block.  We construct such a field while keeping
it divergence-free and nowhere zero. We observe that the outer and the inner toroidal surfaces are not nested, in the sense that the region in between is not diffeomorphic to a toroidal annulus. The geometry of the whole block will become clear in the next description.

Fix a sufficiently small constant $q>0$ and put $a:=2q$. As usual, parametrize a rotationally symmetric solid
torus $S_0$ by the embedding
\begin{equation}
    F(r,\varphi,h)
    :=\bigl((1+r\cos\varphi)\cos h,
            (1+r\cos\varphi)\sin h,-r\sin\varphi\bigr),
    \qquad 0\leq r\leq a,
    \label{eq:volume-embedding}
\end{equation}
where both angles lie in $\T:=\R/(2\pi\mathbb Z)$ and $\varphi$
collapses at $r=0$. As costumary in the literature, $h$ is called the toroidal angle and $\varphi$ is the poloidal (or meridional) angle. In suitable Cartesian coordinates, the core of the solid torus $S_0$ is the unit circle on a plane perpendicular to the third coordinate axis. Set $J:=1+r\cos\varphi$. In coordinates $(r,\varphi,h)$ the Euclidean volume and metric are
\begin{equation}
    F^*\vol=Jr\,dr\wedge d\varphi\wedge dh,
    \qquad
    F^*g_{\mathrm E}=dr^2+r^2d\varphi^2+J^2dh^2.
    \label{eq:F-volume}
\end{equation}
Next, we use a tube around a poloidal circle for the smaller (inner) block.
Indeed, using the real-valued representative of the angle $h$ near zero, define
\[
    T_1:=F\left\{(r-q)^2+h^2\leq\frac{q^2}{4}\right\}.
\]
Clearly, this solid torus lies in $\Int S_0$, because $q/2\leq r\leq3q/2<a$
there.  Its core is the poloidal Euclidean circle
\[
    \Lambda_1:=\{(1+q\cos\varphi,0,-q\sin\varphi):
                   \varphi\in\T\}.
\]
Thus $S_0$ is axisymmetric about the third coordinate axis, whereas
$T_1$ is a poloidal tube around the meridional circle $\Lambda_1$;
see Figure~\ref{fig:toroidal-block}.  This is the basic geometric
configuration that we will repeat at smaller scales.

\begin{figure}[tbp]
    \centering
    \includegraphics[width=\linewidth]{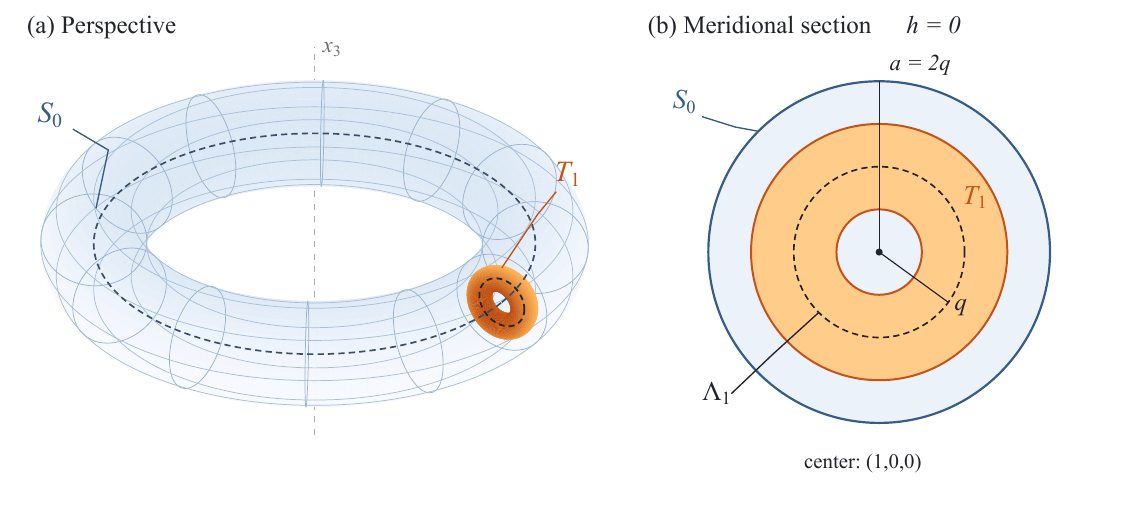}
    \caption{The basic block $(S_0,T_1)$, $T_1\Subset S_0$: perspective view (left)
    and meridional section $h=0$ (right).  Dashed curves mark the cores;
    proportions are schematic.}
    \label{fig:toroidal-block}
\end{figure}

Now we construct the magnetic field in this first block. To this end, choose a smooth cutoff function $0\leq\chi(r,h)\leq1$, periodic in $h$, equal
to one when $(r-q)^2+h^2\leq(3q/5)^2$ and to zero when
$(r-q)^2+h^2\geq(2q/3)^2$.  Its support lies in a single angular
chart and stays away from $r=0$.  Put
\[
    H(r,h):=r^2(1-\chi(r,h)).
\]
Thus $H=r^2$ near the core curve and for $r\geq q+2q/3=5a/6$, while $H=0$ on a
neighborhood of $T_1$. In particular, $H$ is of class $C^\infty$. The following vector field is designed so that these choices preserve
rotational symmetry in a neighborhood of $\partial S_0$ (specifically, on $r\geq 5a/6$) and leave only the tangential
$\varphi$ component near $T_1$:
\begin{equation}
    u:=\partial_\varphi-\frac{H_h}{r}\partial_r
                         +\frac{H_r}{r}\partial_h,
    \qquad
    B_0:=F_*\left(\frac{u}{J}\right).
    \label{eq:toroidal-block-field}
\end{equation}
Indeed, near $r=0$ we have $u=\partial_\varphi+2\partial_h$.
The transverse rotation field $\partial_\varphi$ extends smoothly
and vanishes at the core, where $u=2\partial_h$, thus giving smoothnes of $u$.  For $r>0$, the
$\partial_\varphi$ component of $u$ is one.  Hence $B_0$ is smooth
and nowhere zero on a neighborhood of $S_0$.

The factor $J^{-1}$ in \eqref{eq:toroidal-block-field} allows us to show that $B_0$ is divergence-free.  Indeed,
\begin{equation}
\begin{split}
    F^*\beta_{B_0}
    &=-r\,dr\wedge dh+dH\wedge d\varphi\\
    &=d\left(-\frac{r^2}{2}\,dh+H\,d\varphi\right).
\end{split}
\label{eq:model-magnetic-form}
\end{equation}
Notice that the $1$-form $Hd\varphi-\frac{r^2}{2}dh$ is smooth at the core curve, since $H=r^2$ there and
$r^2d\varphi$ is smooth.  In particular, we deduce that $\Div B_0=0$.

Near the poloidal surface $\partial T_1$ the vector field is
\[
    B_0=F_*\left(\frac{\partial_\varphi}{J}\right),
\]
so it is tangent to all nearby levels of $(r-q)^2+h^2$.  In the
outer toroidal collar $r\geq5a/6$, the vector field has the form
\begin{equation}
    B_0=F_*\left(\frac{\partial_\varphi+2\partial_h}{J}\right).
    \label{eq:outer-collar-u}
\end{equation}
Notice that the vector field $K_0:=F_*\partial_h$ is Killing, and generates Euclidean rotations about the
third coordinate axis.  By Equation~\eqref{eq:model-magnetic-form}, the
function 
$$\psi(F(r,\varphi,h)):=-r^2/2$$ satisfies the identity
\begin{equation}
    d\psi=-\iota_{K_0}\beta_{B_0}\,,
    \label{eq:barrier-scalar}
\end{equation}
in the aforementioned outer toroidal collar; in particular, $K_0$ is a Killing symmetry of $B_0$ there.
In view of Lemma~\ref{lem:noether-barrier}, we can now describe the toroidal magnetic barrier, which we will later reproduce at every scale. Set
\begin{equation}
    r_-:=\frac{7a}{8},\qquad r_+:=\frac{15a}{16},
    \qquad W^\pm:=F\{r<r_\pm\}.
    \label{eq:radii-barrier}
\end{equation}
Then $T_1\Subset W^-$, and
$\mathcal C_0:=\overline{W^+}\setminus W^-$ lies strictly inside
the region where $H=r^2$.  Across this collar, $\psi$ has a jump
of absolute value $(r_+^2-r_-^2)/2$, while $\abs{K_0}\leq1+r_+$.
We conclude from \Cref{lem:noether-barrier} that no solution $x(t)$ of the Newton--Lorentz equation with magnetic field $B_0$, lying in the collar $\mathcal C_0$, can cross the outer boundary $\partial W^+$ if its speed is below the following threshold 
\[
    \nu_*:=\frac{r_+^2-r_-^2}{4(1+r_+)}>0.
\]

This step has provided the basic construction that we will glue smoothly at smaller scales to create a magnetic field yielding a Lyapunov stable zero-speed equilibrium for the Newton--Lorentz equation.

\subsubsection*{Step 2: Gluing nested toroidal blocks}

The purpose of this step is to place a smaller copy of the entire construction of Step~1 inside the poloidal tube $T_1$. This smaller copy will be denoted by $S_1$, and the region between 
$S_1$ and $T_1$ will be filled with a suitable transition vector field.
To keep the eventual zero isolated, the field joining the two blocks
must be nowhere zero as well as divergence-free.  We arrange this by
placing the smaller block so that both boundary fields have a positive
component along the same angular variable in suitable product coordinates.

To begin with the construction, consider the orientation-preserving similarity
\begin{equation}
    g(x_1,x_2,x_3):=(1+qx_1,qx_3,-qx_2).
    \label{eq:g-similarity}
\end{equation}
Here $q$ is the small constant fixed in the previous step. It is easy to check that $g$ maps the core of $S_0$ onto the poloidal circle $\Lambda_1$, it has a contraction ratio
$q$, and has a fixed point $O=(1/(1-q),0,0)$. Moreover, for small $q$, we claim that
\begin{equation}
    S_1:=g(S_0)\Subset T_1\Subset S_0.
    \label{eq:nesting-first-step}
\end{equation}
To check this statement and to construct coordinates between the blocks, we
rescale the embedding $F$ by $g^{-1}$ and introduce rescaled variables $(\xi,\varphi,\zeta)$ as 
$$r=q(1+\xi)\,,\qquad \varphi=\varphi\,,\qquad h=q\zeta\,.$$  The map
\begin{equation}
    G_q(\xi,\varphi,\zeta)
    :=g^{-1}F(q(1+\xi),\varphi,q\zeta)
    \label{eq:transition-rescaling}
\end{equation}
parametrizes $g^{-1}(T_1)$ on $\xi^2+\zeta^2\leq1/4$.
Taylor expansion gives $G_q=G_0+O(q)$ in every fixed $C^m$ norm on
$\xi^2+\zeta^2\leq(3/5)^2$, where
\[
    G_0(\xi,\varphi,\zeta)
    =((1+\xi)\cos\varphi,(1+\xi)\sin\varphi,\zeta).
\]
Thus the outer toroidal surface of $g^{-1}(T_1)$ is a small deformation of a standard torus.
Next, we interpolate its parametrization with that of the inner torus $\partial S_0$,
keeping both unchanged near these boundaries.  Set
\[
    \Phi_0(s,\alpha,\gamma)
    :=G_0(s\cos\gamma,\alpha,s\sin\gamma),
\]
and choose a fixed $\eta\in C_c^\infty((1/3,3/5))$, with
$0\leq\eta\leq1$ and $\eta=1$ near $s=1/2$.
For $0<s<3/5$, define the interpolated parametrization
\[
    \Phi_q(s,\alpha,\gamma)
    :=(1-\eta(s))\Phi_0(s,\alpha,\gamma)
       +\eta(s)G_q(s\cos\gamma,\alpha,s\sin\gamma).
\]
Let us check that $\Phi_q$ is, indeed, a global parametrization. Since the cutoff is supported away from the core and the outer
edge of this coordinate domain, $\Psi_q:=\Phi_q\circ\Phi_0^{-1}$ extends
by the identity to a map $\Id+h_q$ of $\R^3$, with
$\norm{Dh_q}_{L^\infty}=O(q)<1$.  This is an orientation-preserving
diffeomorphism, and we can solve $x=y-h_q(x)$ by contraction to obtain its inverse. Moreover,
$\Psi_q$ fixes $S_0$, since $a=2q<1/3$, and maps the solid torus
$s\leq1/2$ onto $g^{-1}(T_1)$, as their boundaries correspond.
This proves the nesting property~\eqref{eq:nesting-first-step} and shows that restricting
$\Phi_q$ to $a\leq s\leq1/2$ gives a global parametrization, i.e., it defines product coordinates on
\begin{equation}
    P:=g^{-1}(T_1)\setminus\Int S_0
       \simeq[a,1/2]_s\times\T^2_{\alpha,\gamma}.
    \label{eq:transition-coordinate}
\end{equation}
In these coordinates the Euclidean volume is
\[
    \vol=\rho\,ds\wedge d\alpha\wedge d\gamma
\]
for some positive smooth function $\rho$. Moreover, we can endow the set $P$ with the vector field $(\Phi_q)_*(u/J)$, where $u$ is the vector field constructed in the previous step.
Near the inner boundary of $P$ we have $\alpha=h$ and $\gamma=-\varphi$, and
near the outer boundary it holds that $\alpha=\varphi$.  Hence the aforementioned vector field has the following exact expressions near the boundaries of $P$:
\[
\begin{aligned}
    B_0&=\frac{-\partial_\gamma+2\partial_\alpha}
                  {1+s\cos\gamma}
       &&\text{near }\partial S_0,\\
    (g^{-1})_*B_0
       &=\frac{\partial_\alpha}
                  {1+q(1+s\cos\gamma)\cos\alpha}
       &&\text{near }\partial g^{-1}(T_1).
\end{aligned}
\]
Both are tangent to the torus slices $\{s=const\}$ and have positive $\alpha$
component.  This common positive component prevents cancellation
during interpolation.  To preserve divergence as well, we interpolate
the components weighted by the volume density, as in the following lemma.

\begin{lemma}
\label[lemma]{lem:solenoidal-gluing}
Let $W=[s_-,s_+]\times\T^2_{\alpha,\gamma}$ carry a smooth
positive volume form
\[
    \Omega=\rho(s,\alpha,\gamma)\,
             ds\wedge d\alpha\wedge d\gamma.
\]
Let $Y_-$ and $Y_+$ be smooth divergence-free fields on boundary
collars, tangent to all $s$ slices there, with
$d\alpha(Y_\pm)>0$.  For every $c_-,c_+>0$, there is a smooth
divergence-free nowhere-zero field $Y_{c_-,c_+}$ agreeing with $c_-Y_-$
and $c_+Y_+$ on smaller boundary collars.  The construction is linear
in $(c_-,c_+)$ and, for fixed geometry,
\begin{equation}
    \norm{Y_{c_-,c_+}}_{C^m(W)}\leq C_m(c_-+c_+)
    \qquad(m\geq0),
    \label{eq:gluing-estimate}
\end{equation}
where $C_m$ is independent of $c_-,c_+$.
\end{lemma}

\begin{proof}
Since the vector fields are tangent to the $s$ slices, write them on their
respective collars as
\[
    Y_\pm=Y_\pm^\alpha\partial_\alpha
                 +Y_\pm^\gamma\partial_\gamma,
    \qquad
    b_\pm^j:=\rho Y_\pm^j\quad(j=\alpha,\gamma).
\]
For any tangential field $Y$, the divergence with respect to the volume form $\Omega$ is
\[
    \Div_\Omega Y=\frac1\rho\left[
        \partial_\alpha(\rho Y^\alpha)
        +\partial_\gamma(\rho Y^\gamma)\right].
\]
Consequently, the weighted boundary components satisfy
\begin{equation}
    \partial_\alpha b_\pm^\alpha
        +\partial_\gamma b_\pm^\gamma=0,
    \label{eq:weighted-divergence}
\end{equation}
and $b_\pm^\alpha>0$.  The absence of an $s$ derivative is what
allows us to extend these pairs by freezing the $s$ variable inside
each collar.

Choose $\delta>0$ so that $4\delta<s_+-s_-$ and the given fields
are defined on $[s_-,s_-+2\delta]\times\T^2$ and
$[s_+-2\delta,s_+]\times\T^2$, respectively.  Take smooth functions
$0\leq\zeta_\pm(s)\leq1$ on $[s_-,s_+]$ with
\[
\begin{aligned}
    \zeta_-&=1\quad\text{for }s\leq s_-+\delta,
    &\zeta_-&=0\quad\text{for }s\geq s_-+2\delta,\\
    \zeta_+&=1\quad\text{for }s\geq s_+-\delta,
    &\zeta_+&=0\quad\text{for }s\leq s_+-2\delta.
\end{aligned}
\]
Define
\[
\begin{aligned}
    \sigma_-(s)&:=\zeta_-(s)s
                    +(1-\zeta_-(s))(s_-+\delta),\\
    \sigma_+(s)&:=\zeta_+(s)s
                    +(1-\zeta_+(s))(s_+-\delta).
\end{aligned}
\]
The map $\sigma_-$ takes values in $[s_-,s_-+2\delta]$, equals
$s$ for $s\leq s_-+\delta$, and is constant for
$s\geq s_-+2\delta$.  The analogous properties hold for $\sigma_+$
at the other endpoint.  We can therefore define smooth weighted
components on all of $W$ by
\[
    w_\pm^j(s,\alpha,\gamma)
        :=b_\pm^j(\sigma_\pm(s),\alpha,\gamma),
    \qquad j=\alpha,\gamma.
\]
Because $\sigma_\pm$ depends only on $s$, equation
\eqref{eq:weighted-divergence} gives
\[
    \partial_\alpha w_\pm^\alpha
       +\partial_\gamma w_\pm^\gamma
    =\bigl(\partial_\alpha b_\pm^\alpha
       +\partial_\gamma b_\pm^\gamma\bigr)
           (\sigma_\pm(s),\alpha,\gamma)=0.
\]
Also $w_\pm^\alpha>0$ everywhere.  Dividing by the volume density
at the current point, we obtain the extensions
\[
    \widetilde Y_\pm
       :=\frac1{\rho(s,\alpha,\gamma)}
          \left(w_\pm^\alpha\partial_\alpha
                         +w_\pm^\gamma\partial_\gamma\right).
\]
They are divergence-free and agree with $Y_\pm$ on the smaller
collars where $\sigma_\pm(s)=s$.  Notice that it is the weighted
components $\rho Y_\pm^j$ that we freeze; the final denominator is
$\rho(s,\alpha,\gamma)$, rather than the density at
$\sigma_\pm(s)$.

To glue these extensions, choose a smooth $0\leq\vartheta(s)\leq1$
equal to one for $s\leq s_-+\delta$ and zero for
$s\geq s_+-\delta$.  Set
\begin{equation}
    Y_{c_-,c_+}:=c_-\vartheta(s)\widetilde Y_-
                   +c_+(1-\vartheta(s))\widetilde Y_+.
    \label{eq:gluing-form}
\end{equation}
The cutoff depends only on $s$ and both extensions have zero $s$
component, so no derivative of $\vartheta$ contributes to the
divergence.  Thus $\Div_\Omega Y_{c_-,c_+}=0$.  Moreover,
\[
    d\alpha(Y_{c_-,c_+})
       =\frac{c_-\vartheta(s)w_-^\alpha
                 +c_+(1-\vartheta(s))w_+^\alpha}{\rho}>0,
\]
since $c_-,c_+>0$.  The field is therefore nowhere zero.  On the
inner collar it equals $c_-Y_-$, and on the outer collar it equals
$c_+Y_+$, as required.

Finally, the extensions and the cutoff are fixed independently of
$c_-,c_+$.  Formula~\eqref{eq:gluing-form} is linear in these
amplitudes, and the triangle inequality gives
\[
\begin{split}
    \norm{Y_{c_-,c_+}}_{C^m(W)}
    &\leq c_-\norm{\vartheta\widetilde Y_-}_{C^m(W)}
       +c_+\norm{(1-\vartheta)\widetilde Y_+}_{C^m(W)}\\
    &\leq C_m(c_-+c_+).
\end{split}
\]
Here $C_m$ depends on the fixed volume density, boundary fields,
extensions and cutoff, but not on the amplitudes.  This proves
\eqref{eq:gluing-estimate}.
\end{proof}

We can now apply this lemma on $P$ with inner field $B_0$ of amplitude $d$ (i.e., rescaled by a constant factor $d$) and
outer field $(g^{-1})_*B_0$ of amplitude $c$, and push forward by the orientation-preserving similarity
$g$.  We thus obtain a vector field $E_{c,d}$ on
$W_0:=T_1\setminus\Int S_1$ which agrees with $cB_0$ near
$\partial T_1$ and with $d\,g_*B_0$ near $\partial S_1$.
It is nowhere zero for $c,d>0$, and
\begin{equation}
    \norm{E_{c,d}}_{C^m(W_0)}\leq C_m(c+d).
    \label{eq:transition-estimate}
\end{equation}
It is also divergence-free, since push-forward by a similarity
preserves this property: its Jacobian determinant is constant.
The transition lies inside $T_1$, so the outer barrier collar (near $\partial S_0$) is unchanged.

\subsubsection*{Step 3: Iteration and smoothness}

The preceding step allows independent positive amplitudes on successive
blocks.  As before, fix $0<q<1$ small enough for the construction; the geometry
and its constants $C_m$ are now fixed.  The idea is to repeat the gluing at every
scale and choose the amplitudes to ensure smoothness at the limiting point.

To this end, for $k\geq0$, set
\begin{equation}
    S_k:=g^k(S_0),\qquad T_{k+1}:=g^k(T_1).
    \label{eq:nested-tori}
\end{equation}
It is then clear that $S_k\Supset T_{k+1}\Supset S_{k+1}$.  Since $g$ is a
contraction, $\bigcap_{k\geq0}S_k=\{O\}$; the strict nesting
also gives $O\in\Int S_k$ for every $k$. Recall that $O$ is the (unique) fixed point of the map $g$, see Step~2.

To compensate for the growth of derivatives under rescaling, choose
amplitudes that decay faster than every power of the spatial scale $q^k$, e.g., take
\begin{equation}
    \lambda_k:=q^{k^2}>0.
    \label{eq:smooth-amplitudes}
\end{equation}
On the main blocks define the vector field
\begin{equation}
    B:=\lambda_k(g^k)_*B_0
    \quad\text{on }D_k:=S_k\setminus\Int T_{k+1},
    \label{eq:main-block-field}
\end{equation}
and fill the transition regions by the vector field
\begin{equation}
    B:=(g^k)_*E_{\lambda_k,\lambda_{k+1}}
    \quad\text{on }W_k:=T_{k+1}\setminus\Int S_{k+1}.
    \label{eq:transition-field}
\end{equation}
It is straightforward to check that the definitions agree on neighborhoods of every interface.  Thus they
give a smooth divergence-free vector field $B$ on a neighborhood of
$S_0\setminus\{O\}$, nowhere zero there.

Under push-forward by $g^k$, vectors acquire a factor $q^k$ and
each spatial derivative a factor $q^{-k}$.  Since
$\lambda_{k+1}\leq\lambda_k$, \eqref{eq:transition-estimate} gives
\begin{equation}
    \norm{D^mB}_{L^\infty(S_k\setminus\Int S_{k+1})}
    \leq C_m q^{k^2+k(1-m)}.
    \label{eq:main-derivative-estimate}
\end{equation}
The fixed point lies in $\Int S_1$, so there are constants $c,C>0$
such that
\begin{equation}
    cq^k\leq\abs{x-O}\leq Cq^k
    \quad(x\in S_k\setminus\Int S_{k+1}).
    \label{eq:annular-distance}
\end{equation}
For fixed $m,N$, the bound $q^{k^2+k(1-m)}\leq q^{kN}$ holds
for all sufficiently large $k$.  The preceding estimates therefore give,
for all $m,N\geq0$,
\begin{equation}
    \norm{D^mB(x)}\leq C_{m,N}\abs{x-O}^N
    \qquad(x\to O).
    \label{eq:flat-estimate}
\end{equation}
Setting $B(O)=0$ therefore gives a smooth flat extension $B$.  Its
divergence vanishes at $O$ as well, since all first derivatives do.
Since $B$ is nowhere zero in $S_0\setminus\{O\}$, this magnetic
zero is isolated.

\subsubsection*{Step 4: Stability and compact support}

It remains to use the magnetic barriers retained at every small scale.  Set
$W_k^\pm:=g^k(W^\pm)$ and $\mathcal C_k:=g^k(\mathcal C_0)$ (recall the definitions of $W^\pm$ and $\mathcal C_0$ at the end of Step~1).
The gluing leaves these collars unchanged (that is, the interpolation in the gluing process does not change the push-forward fields in these sets), so on $\mathcal C_k$ we
have $B=\lambda_k(g^k)_*B_0$.  The crucial property is that the threshold $\nu_*$ for the model field in the first block (see the end of Step~1) transfers
by rescaling space and time.  The reason is that $g^k$ is an orientation-preserving Euclidean similarity, which transforms Killing fields into Killing fields, and preserves the Newton--Lorentz equation of motion, up to suitable rescalings. More precisely, since the similarity ratio of $g^k$ is $q^k$, a trajectory $x(t)$ in the
collar $\mathcal C_k$ is transformed by
\[
    y(s):=g^{-k}\left(x\left(\frac{s}{\lambda_kq^k}\right)\right)
\]
into a trajectory of the model field in $\mathcal C_0$.  Indeed,
substitution in Newton--Lorentz's equation~\eqref{eq:magnetic-equation} gives
\[
    y''=y'\times B_0(y),
    \qquad
    \abs{y'}=\frac{\abs{\dot x}}{\lambda_kq^{2k}}.
\]
We can then conclude from the magnetic barrier property established in Step~1 that a particle starting in $W_k^-$ cannot leave $W_k^+$ if its
speed is less than
\begin{equation}
    \nu_k:=\lambda_kq^{2k}\nu_*>0.
    \label{eq:speed-threshold}
\end{equation}
Only the trajectory segment crossing the collar is rescaled;
no symmetry of the transition fields is needed.
For stability, it suffices to choose the enclosing barrier first
and then take the initial speed below its threshold, even though
$\nu_k\to0$ as $k\to\infty$.
Indeed, given $\eps>0$, choose $k$ so large that
$W_k^+\subset\mathbb B_{\eps/2}(O)$.  Since
$O\in S_{k+1}\subset T_{k+1}\subset W_k^-$, we can choose
\begin{equation}
    0<\delta<\frac12\min\{
        \dist(O,\R^3\setminus W_k^-),\nu_k,\eps\}.
    \label{eq:stability-delta}
\end{equation}
If $\abs{x(0)-O}+\abs{\dot x(0)}<\delta$, the barrier and speed
conservation imply
\[
    x(t)\in W_k^+,
    \qquad
    \abs{x(t)-O}+\abs{\dot x(t)}<\eps
    \quad(t\geq0).
\]
The trajectory stays in a compact region where $B$ is smooth, with
bounded velocity, so the ODE continuation theorem gives existence for
all these times.  This proves Lyapunov stability at $(O,0)$.

To make the vector field $B$ compactly supported, we simply invoke a standard extension theorem.
Choose a ball $\overline{U_1}\subset \Int S_0$ centered at $O$. Then there exists~\cite{EPP} a divergence-free vector field $\widetilde B\in C_c^\infty(\R^3;\R^3)$ that
agrees with $B$ on $U_1$.  By construction, it retains the isolated flat zero and all sufficiently small barrier collars.  The stability proof therefore
applies unchanged, using only those collars.

Finally, writing $U_1=\mathbb B_\rho(O)$, set
$\widehat B(x):=\widetilde B(O+\rho x)$.  This field is compactly
supported, divergence-free and flat at the origin, and is nonzero for
$0<\abs x<1$.  The change of variables $y(t)=(x(t)-O)/\rho$ sends
trajectories for $\widetilde B$ to trajectories for $\widehat B$ and
preserves Lyapunov stability.  Thus $\widehat B$ satisfies all the
requirements of \Cref{thm:smooth-counterexample}. The proof is complete.

\section*{Acknowledgements}

The authors are grateful to Rafael de la Llave for his interest in this problem and for sending us the sketch of proof in Appendix~\ref{sec:nonzero-field}. This work has received funding from the European Research Council (ERC)
under the European Union's Horizon 2020 research and innovation programme
through grant agreement 862342 (A.E.).  The authors are also partially
supported by the grant PID2022-136795NB-I00 of the Spanish Science Agency and
the ICMAT--Severo Ochoa grant CEX2023-001347-S.

\section*{Declaration of use of AI tools}

The fundamental ideas of this work are due to the authors. During the
preparation of the article, the authors used ChatGPT to assist
with the exposition and organization of the manuscript and with checking
technical details. The authors take full responsibility for the
mathematical arguments and the final content of the article.

\appendix

\section{Proof of Theorem~\ref{thm:away-from-magnetic-zeros}}
\label{sec:nonzero-field}

Set $b:=\abs{B(0)}>0$ and $n_0:=B(0)/b$.  Choose $R>0$ with
$\overline{\mathbb B_R(0)}\subset U$, and put
\[
    M:=\sup_{\mathbb B_R(0)}\norm{DB}.
\]
Fix $0<r<R$ so that $4Mr\leq b$.  For $\eps>0$, consider the
solution with $x(0)=0$ and $\dot x(0)=\eps n_0$.  Speed conservation
gives $\abs{\dot x(t)}=\eps$ and $\abs{x(t)}\leq\eps t$.
In particular, the solution stays in $\overline{\mathbb B_r(0)}$
for $0\leq t\leq r/\eps$ and exists throughout this interval by
the ODE continuation theorem.

The Lorentz force is orthogonal to $B$, so
\[
    \frac{d}{dt}\bigl(\dot x\cdot B(x)\bigr)
    =\dot x\cdot DB(x)\dot x\geq-M\eps^2.
\]
Since $\dot x(0)\cdot B(0)=\eps b$ and
$\abs{B(x(t))-B(0)}\leq M\eps t$, it follows that
\[
\begin{split}
    \dot x(t)\cdot n_0
    &=\frac1b\bigl(\dot x(t)\cdot B(x(t))
        +\dot x(t)\cdot[B(0)-B(x(t))]\bigr)\\
    &\geq\eps-\frac{2M\eps^2t}{b}
     \geq\frac\eps2
     \qquad(0\leq t\leq r/\eps).
\end{split}
\]
Integrating yields
\[
    \abs{x(r/\eps)}\geq x(r/\eps)\cdot n_0\geq\frac r2.
\]
The escape distance $r/2$ is independent of $\eps$, which can be
arbitrarily small.  This proves Lyapunov instability.

\makeatletter
\renewcommand{\@biblabel}[1]{#1.}
\makeatother

\end{document}